\UseRawInputEncoding
\documentclass[11pt]{article}
\usepackage{amsmath,amssymb,amsthm,amsfonts}
\usepackage{graphicx}
\usepackage{float}
\usepackage[scriptsize,nooneline]{subfigure}
\usepackage{indentfirst}
\usepackage{cite}
\usepackage{color}
\usepackage{mathrsfs}
\usepackage{epstopdf}
\usepackage{setspace}
\usepackage[labelfont=footnotesize,font=footnotesize]{caption}
\usepackage[colorlinks, citecolor=red]{hyperref}

\newtheorem{theorem}{Theorem}[section]

\newtheorem{lemma}{Lemma}[section]

\usepackage{hyperref}
\def\geq{\geqslant}\def\leq{\leqslant}

\begin{document}
\title{\bf A characterization of mixed $\lambda$-central $\mathrm{BMO}$ space via the commutators of Hardy type operators}
\author{Wenna Lu and Jiang Zhou\thanks{The research is supported by the NNSF of China (No. 12061069).} \\[.5cm]}

\date{}
\maketitle

{\bf Abstract.} In this paper, we give a characterization of mixed $\lambda$-central bounded mean oscillation space $\mathrm{CBMO}^{\vec{q},\lambda}(\mathbb{R}^{n})$ via the boundedness of the commutators of $n$-dimensional Hardy operator $\mathcal{H}$ and its dual operator $\mathcal{H}^{*}$ on mixed Lebesgue space $L^{\vec{p}}(\mathbb{R}^{n})$. In addition, we also establish the boundedness of commutators $\mathcal{H}_{b}$ and $\mathcal{H}^{*}_{b}$ generated with $\mathrm{CBMO}^{\vec{q},\lambda}(\mathbb{R}^{n})$ function $b$ on mixed $\lambda$-central Morrey space $\mathcal{B}^{\vec{q},\lambda}(\mathbb{R}^{n})$, respectively.


{\bf MSC(2020) subject classification:} 42B20, 42B25, 42B35.

\baselineskip 15pt
\section{Introduction}
\quad Let $f$ be a non-negative integrable function on $\mathbb{R}^{+}$. The classical Hardy operator and its dual operator are defined by
$$Hf(x):=\frac{1}{x}\int^{x}_{0}f(t)dt, \qquad H^{*}f(x):=\int^{\infty}_{x}\frac{f(t)}{t}dt, \quad x>0.$$
The Hardy operators $H$ and $H^{*}$ are adjoint mutually:
$$\int^{\infty}_{0}(Hf)(x)g(x)dx=\int^{\infty}_{0}f(x)(H^{*}g)(x)dx,$$
when $f\in L^{p}(\mathbb{R}^{+}), g\in L^{q}(\mathbb{R}^{+}), 1<p<\infty, 1/p+1/q=1$.

In 1920, Hardy \cite{H1} first introduced the classical Hardy operators, and established the most celebrated Hardy's integral inequality as follows:
$$\|Hf\|_{L^{p}}\leq\frac{p}{p-1}\|f\|_{L^{p}}, \qquad \|H^{*}f\|_{L^{q}}\leq\frac{p}{p-1}\|f\|_{L^{q}}.$$
Moreover,
$$\|H\|_{L^{p}\rightarrow L^{p}}=\|H^{*}\|_{L^{q}\rightarrow L^{q}}=\frac{p}{p-1},$$
where $1/p+1/q=1$.

In 1976, the classical Hardy operator was extended to $n$-dimensional form by Faris \cite{F1}, that is, let $f$ be a locally integrable function in $\mathbb{R}^{n}$, the $n$-dimensional Hardy operator $\mathcal{H}$ is defined by
$$\mathcal{H}f(x)=\frac{1}{|x|^{n}}\int_{|t|<|x|}f(t)dt, \quad x\in\mathbb{R}^{n}\backslash\{0\},$$
The norm of $\mathcal{H}$ on $L^{p}(\mathbb{R}^{n})$ was evaluated in \cite{CG} and found to be equal to that of the one-dimensional Hardy operator in \cite{HLP}, i.e.,
$$\frac{p}{p-1}|B(0,1)| \quad (1<p<\infty).$$
Similarly, the dual operator $\mathcal{H}^{*}$ of $\mathcal{H}$ is defined by
$$\mathcal{H}^{*}f(x)=\int_{|t|\geq|x|}\frac{f(t)}{|t|^{n}}dt, \quad x\in\mathbb{R}^{n}\backslash\{0\}.$$
It means that $\mathcal{H}$ and $\mathcal{H}^{*}$ satisfy
$$\int_{\mathbb{R}^{n}}(\mathcal{H}f)(x)g(x)dx=\int_{\mathbb{R}^{n}}f(x)(\mathcal{H}^{*}g)(x)dx,$$
when $f\in L^{p}(\mathbb{R}^{n}), g\in L^{q}(\mathbb{R}^{n}), 1<p<\infty, 1/p+1/q=1$.

In recent years, Hardy operator, as one of the important average operators, has increasing considerable attention and much developments. Furthermore, the commutators of Hardy operators are also an important subject in the study of Hardy type operators. Recall that for a locally integrable function $b$ and a Calder\'{o}n-Zygmund singular integral operator $T$, a well-known result of Coifman et al. \cite{CRW} stated that the commutator $T_{b}f=bTf-T(bf)$ is bounded on $L^{p}(\mathbb{R}^{n})$ if and only if the symbol function $b$ is in the bounded mean oscillation space $\mathrm{BMO}(\mathbb{R}^{n})$. An important work on the commutator of Hardy operator is due to Long and Wang \cite{LW}, who showed that the commutators of Hardy operator and its dual operator are bounded on the Lebesgue spaces if and only if the symbol function $b$ belong to the central bounded mean oscillation space, that is, a bigger space than the classical bounded mean oscillation space $\mathrm{BMO}(\mathbb{R}^{n})$. These facts indicated that the commutators of Hardy type operators are fundamentally different from those of the Calder\'{o}n-Zygmund singular integral operators and other important operators. The difference happens because the Hardy operator is an average operator around the origin, whereas the Calder\'{o}n-Zygmund singular integral operator has no such property. In 2007, Fu et al. \cite{FLW} gave the characterization of the central bounded mean oscillation space $\mathrm{CBMO}^{p}(\mathbb{R}^{n})$ via the boundedness of the commutator of Hardy type operators on different function spaces. In 2013, Zhao and Lu \cite{ZL} proved that the commutators of $n$-dimensional Hardy operator and its dual operator are bounded on Lebesgue space $L^{p}(\mathbb{R}^{n})$ if and only if the symbol function $b$ is in the $\lambda$-central bounded mean oscillation space $\mathrm{CBMO}^{q,\lambda}(\mathbb{R}^{n})$. Moreover, many other estimates of commutators of Hardy type operators on classical function spaces were obtained in \cite{FL,KY,LW,SL,ZFL}.

In 1961, Benedek and Panzone \cite{BP} introduced the mixed Lebesgue space and established some elementary properties. In fact, the more accurate structure of the mixed-norm function spaces than the corresponding classical function spaces such that the mixed-norm function spaces have more extensive applications in various areas such as potential analysis, harmonic analysis and PDEs, see \cite{AI,KC,KD,KN} and the references therein.
In particular, Dong et al. \cite{DKP} considered Stokes systems with measurable coefficients and Lions-type boundary conditions by virtue of mixed function spaces. Therefore, the topic of function spaces with the mixed norm has increasing considerable attention and much developments in recent years. In 2019, Nogayama \cite{NT1}, to generalize Morrey space and mixed Lebesgue space, introduced the mixed Morrey space, and proved the necessary and sufficient conditions for boundedness of commutators of fractional integral operators on mixed Morrey space in \cite{NT2}. On the other hand, Alvarez et al. \cite{AL} introduced the $\lambda$-central bounded mean oscillation space and $\lambda$-central Morrey space in 2000. After that, Fu et al. \cite{FLL} and Yu et al. \cite{YT} obtained $\lambda$-central $\mathrm{BMO}$ estimates for commutators of some classical operators. In 2022, Wei \cite{W1} introduced the mixed central bounded mean oscillation space $\mathrm{CBMO}^{\vec{q}}(\mathbb{R}^{n})$, and gave a characterization of $\mathrm{CBMO}^{\vec{q}}(\mathbb{R}^{n})$ via the boundedness of the commutator of Hardy type operators on mixed Herz space. In the same year, Lu and Zhou \cite{LZ} introduced mixed $\lambda$-central bounded mean oscillation space $\mathrm{CBMO}^{\vec{q},\lambda}(\mathbb{R}^{n})$ and mixed $\lambda$-central Morrey space $\mathcal{B}^{\vec{q},\lambda}(\mathbb{R}^{n})$, and established the boundedness of the fractional integral operator and its commutator on this space.

The purpose of this paper will be to obtain a characterization of mixed $\lambda$-central bounded mean oscillation space $\mathrm{CBMO}^{\vec{q},\lambda}(\mathbb{R}^{n})$ via the boundedness of the commutators of $n$-dimensional Hardy operator $\mathcal{H}$ and its dual operator $\mathcal{H}^{*}$ on mixed Lebesgue space, and establish the boundedness of commutators $\mathcal{H}_{b}$ and $\mathcal{H}^{*}_{b}$ generated with $\mathrm{CBMO}^{\vec{q},\lambda}(\mathbb{R}^{n})$ function $b$ on mixed $\lambda$-central Morrey space $\mathcal{B}^{\vec{q},\lambda}(\mathbb{R}^{n})$, respectively.

This paper is organized as follows. Section 2 will present some basic definitions and preliminaries. In Section 3, by the estimates of commutators of the Hardy operator and its dual operator on mixed Lebesgue space, a characterization of the mixed $\lambda$-central bounded mean oscillation space will be given. Furthermore, we will establish the boundedness the commutators of the Hardy operator and its dual operator on mixed $\lambda$-central Morrey space in Section 4.

As a rule, for any set $E\subset\mathbb{R}^n$, ${\chi}_{E}$ denotes its characteristic function and $E^{c}$ denotes its complementary set, we also denote the Lebesgue measure by $|E|$. Let $\mathscr{M}(\mathbb{R}^{n})$ be the class of Lebesgue measurable functions on $\mathbb{R}^{n}$.  We use the symbol $f\lesssim g$ to denote there exists a positive constant $C$ such that $f\leq Cg $, and the notation $f\thickapprox g$ means that there exist positive constants $C_1, C_2$ such that $C_1 g\leq f\leq C_2g$. Throughout this paper, the letter $\vec{p}$ denotes $n$-tuples of the numbers in $(0,\infty]$, $n\geq1$, $\vec{p}=(p_{1},p_{2}, \cdots,p_{n})$, $0<\vec{p}\leq\infty$ means $0<p_{i}\leq\infty$ for all $i=1,2,\cdots,n$. Furthermore, for $1\leq\vec{p}\leq\infty$, let $$\frac{1}{\vec{p}}=(\frac{1}{p_{1}},\frac{1}{p_{2}},\cdots,\frac{1}{p_{n}}), \quad\vec{p}'=(p_{1}',p_{2}',\cdots,p_{n}'),$$
where $p_{i}'=\frac{p_{i}}{p_{i}-1} (i=1,2,\cdots,n)$ is the conjugate exponent of $p_i$.

\section{Some Definitions and Preliminaries}

In this section, we first give the definitions of commutators of the Hardy operator and its dual operator, and then recall the definitions of mixed function spaces as they relate to this paper.

The definitions of commutators of $n$-dimensional Hardy operators are given by Fu et al. \cite{FLW} in 2007. That is, let $b$ be a locally integrable function in $\mathbb{R}^{n}$, The commutators are defined by
$$\mathcal{H}_{b}f=b\mathcal{H}f-\mathcal{H}(fb)$$
and
$$\mathcal{H}^{*}_{b}f=b\mathcal{H}^{*}f-\mathcal{H}^{*}(fb)$$

Next, we recall the definition of mixed Lebesgue space $L^{\vec{p}}(\mathbb{R}^{n})$.

\hspace{-22pt}{\bf Definition 2.1}(\cite{BP}) {For $0<\vec{p}<\infty$, the mixed Lebesgue norm $\|\cdot\|_{L^{\vec{p}}(\mathbb{R}^{n})}$ is defined by
$$\|f\|_{L^{\vec{p}}(\mathbb{R}^{n})}:=\|f\|_{\vec{p}}:=\Bigg(\int_{\mathbb{R}}\cdots\bigg(\int_{\mathbb{R}}\big(\int_{\mathbb{R}}|f(x_{1},x_{2},\cdots,x_{n})|^{p_{1}}
dx_{1}\big)^{\frac{p_{2}}{p_{1}}}dx_{2}\bigg)^{\frac{p_{3}}{p_{2}}}\cdots dx_{n}\Bigg)^{\frac{1}{p_{n}}},$$
where $f:\mathbb{R}^{n}\rightarrow\mathbb{C}$ is a measurable function. If $p_{i}=\infty$, for some $i=1,\cdots,n$, we only need to make some suitable modifications. We define the mixed Lebesgue space $L^{\vec{p}}(\mathbb{R}^{n})$ to be the set of all $f\in\mathscr{M}(\mathbb{R}^{n})$ with $\|f\|_{L^{\vec{p}}(\mathbb{R}^{n})}<\infty$.
}

\quad\hspace{-22pt}{\bf Remark 2.1} {The mixed Lebesgue space $L^{\vec{p}}(\mathbb{R}^{n})$ reduces to the classical Lebesgue spaces
$L^{p}(\mathbb{R}^{n})$ when $p_{1}=p_{2}=\cdots=p_{n}=p$.
}

Before stating our results, let us recall the definitions of mixed $\lambda$-central bounded mean oscillation space $\mathrm{CBMO}^{\vec{q},\lambda}(\mathbb{R}^{n})$ and mixed $\lambda$-central Morrey space $\mathcal{B}^{\vec{q},\lambda}(\mathbb{R}^n)$.

\hspace{-22pt}{\bf Definition 2.2}{\it \; (\cite{LZ}) Let $1<\vec{q}<\infty$ and $\lambda<\frac{1}{n}$,
the mixed $\lambda$-central bounded mean oscillation space $\mathrm{CBMO}^{\vec{q},\lambda}(\mathbb{R}^{n})$ is defined by
$$\mathrm{CBMO}^{\vec{q},\lambda}(\mathbb{R}^{n}):=\{f\in \mathscr{M}(\mathbb{R}^{n}):\|f\|_{\mathrm{CBMO}^{\vec{q},\lambda}(\mathbb{R}^{n})}<\infty\},$$
where
$$\|f\|_{\mathrm{CBMO}^{\vec{q},\lambda}(\mathbb{R}^{n})}:=
\sup\limits_{r>0}\frac{\|(f-f_{B(0,r)})\chi_{B(0,r)}\|_{L^{\vec{q}}(\mathbb{R}^{n})}}{|B(0,r)|^{\lambda}\|\chi_{B(0,r)}\|_{L^{\vec{q}}(\mathbb{R}^{n})}}$$
and
$$f_{B(0,r)}:=\frac{1}{|B(0,r)|}\int_{B(0,r)}f(t)dt.$$
}

\quad\hspace{-22pt}{\bf Remark 2.2} {If $\lambda=0$, then we can easily get $\mathrm{CBMO}^{\vec{q}}(\mathbb{R}^{n})$ was given in \cite{W1}. The space $\mathrm{CBMO}^{\vec{q},\lambda}(\mathbb{R}^{n})$ is just the space $\mathrm{CBMO}^{q,\lambda}(\mathbb{R}^{n})$ was given in \cite{AL} when $\vec{q}=q$.
}

\hspace{-22pt} {\bf Definition 2.3}{\it \; (\cite{LZ}) Let $1<\vec{q}<\infty$ and $\lambda\in\mathbb{R}$,
the mixed $\lambda$-central Morrey space $\mathcal{B}^{\vec{q},\lambda}(\mathbb{R}^n)$ is defined by
$$\mathcal{B}^{\vec{q},\lambda}(\mathbb{R}^n):=\{f\in \mathscr{M}(\mathbb{R}^{n}):\|f\|_{\mathcal{B}^{\vec{q},\lambda}(\mathbb{R}^n)}<\infty\},$$
where
$$\|f\|_{\mathcal{B}^{\vec{q},\lambda}(\mathbb{R}^n)}:=\sup\limits_{r>0}\frac{\|f\chi_{B(0,r)}\|_{L^{\vec{q}}(\mathbb{R}^{n})}}{|B(0,r)|^{\lambda}\|\chi_{B(0,r)}\|_{L^{\vec{q}}(\mathbb{R}^{n})}}.$$
}

\quad\hspace{-22pt}{\bf Remark 2.3} {If $\vec{q}=q$, then the space $\mathcal{B}^{\vec{q},\lambda}(\mathbb{R}^n)$ reduces to the space $\mathcal{B}^{q,\lambda}(\mathbb{R}^n)$ was introduced in \cite{AL}.
}

Furthermore, we also needed some important lemmas in the process of proofs.
\begin{lemma}(\cite{W1})
 Let $1<\vec{p}<\infty$, Then
 $$\|\mathcal{H}f\|_{L^{\vec{p}}(\mathbb{R}^{n})}\lesssim\|f\|_{L^{\vec{p}}(\mathbb{R}^{n})}$$
 and
 $$\|\mathcal{H}^{*}f\|_{L^{\vec{p}}(\mathbb{R}^{n})}\lesssim\|f\|_{L^{\vec{p}}(\mathbb{R}^{n})}.$$
\end{lemma}

\begin{lemma}(\cite{BP})
 Let $1<\vec{p}<\infty$, Then $(L^{\vec{p}})^{*}=L^{\vec{p}'}$.
\end{lemma}

\section{A characterization of $\mathrm{CBMO}^{\vec{q},\lambda}(\mathbb{R}^{n})$ via the boundedness of $\mathcal{H}_{b}$ and $\mathcal{H}^{*}_{b}$ on $L^{\vec{p}}(\mathbb{R}^{n})$}

In this section, we obtain a characterization of the mixed $\lambda$-central bounded mean oscillation space $\mathrm{CBMO}^{\vec{q},\lambda}(\mathbb{R}^{n})$ by boundedness of commutators $\mathcal{H}_{b}$ and $\mathcal{H}^{*}_{b}$ of the Hardy operator and its dual operator on mixed Lebesgue space $L^{\vec{p}}(\mathbb{R}^{n})$. Our main results can be formulated as follows.

\begin{theorem}
Let $1<\vec{p},\vec{q}<\infty$, $0\leq\lambda<\frac{1}{n}$ satisfy condition $\lambda=\frac{1}{n}\sum^{n}_{i=1}\frac{1}{p_{i}}-\frac{1}{n}\sum^{n}_{i=1}\frac{1}{q_{i}}$. Then both $\mathcal{H}_{b}$ and $\mathcal{H}^{*}_{b}$ are bounded from $L^{\vec{p}}(\mathbb{R}^{n})$ to $L^{\vec{q}}(\mathbb{R}^{n})$ if and only if
$$b\in \mathrm{CBMO}^{\vec{q},\lambda}(\mathbb{R}^{n})\cap \mathrm{CBMO}^{\vec{p}',\lambda}(\mathbb{R}^{n}).$$
\end{theorem}

\quad\hspace{-22pt}{\bf Remark 3.1} {If $\mathcal{H}_{b}$ and $\mathcal{H}^{*}_{b}$ are bounded from $L^{\vec{p}}(\mathbb{R}^{n})$ to $L^{\vec{q}}(\mathbb{R}^{n})$, then by Lemma 2.2, we can deduce that both $\mathcal{H}_{b}$ and $\mathcal{H}^{*}_{b}$ map from $L^{\vec{q}'}(\mathbb{R}^{n})$ to $L^{\vec{p}'}(\mathbb{R}^{n})$.
}

The proof of Theorem 3.1 is based on the following a key lemma. For this purpose, we provide the details as follows.
Let
$$B_{k}=\{x\in\mathbb{R}^{n}: |x|\leq 2^{k}\}, C_{k}=B_{k}\backslash B_{k-1}, \chi_{k}=\chi_{C_{k}}, k\in\mathbb{Z},$$
where $\chi_{C_{k}}$ is the characteristic function of a set $C_{k}$. We start with this necessary lemma.

\begin{lemma}
 Let $1<\vec{q}<\infty$, $\lambda<\frac{1}{n}$ and $i, k\in\mathbb{Z}$. If $b\in \mathrm{CBMO}^{\vec{q},\lambda}(\mathbb{R}^{n})$, then
 $$|b(y)-b_{B_{k}}|\leq|b(y)-b_{B_{j}}|+C\frac{2^{n\lambda}}{|1-2^{n\lambda}|}|2^{kn\lambda}-2^{jn\lambda}|\|b\|_{\mathrm{CBMO}^{\vec{q},\lambda}(\mathbb{R}^{n})}.$$
\end{lemma}

\begin{proof}[Proof]
By the H\"{o}lder inequality on mixed Lebesgue space (see \cite{BP}), we get
\begin{align*}
|b_{B_{i}}-b_{B_{i+1}}|&\leq\frac{1}{|B_{i}|}\int_{B_{i}}|b(y)-b_{B_{i+1}}|dy\\
&\leq\frac{1}{|B_{i}|}\|\chi_{B_{i+1}}(b-b_{B_{i+1}})\|_{1}\\
&\leq\frac{1}{|B_{i}|}\|\chi_{B_{i+1}}\|_{\vec{q}'}\|(b-b_{B_{i+1}})\chi_{B_{i+1}}\|_{\vec{q}}\\
&\leq\frac{|B_{i+1}|^{\lambda}}{|B_{i}|}\|\chi_{B_{i+1}}\|_{\vec{q}'}\|\chi_{B_{i+1}}\|_{\vec{q}}\|b\|_{\mathrm{CBMO}^{\vec{q},\lambda}(\mathbb{R}^{n})}\\
&\leq C2^{(i+1)n\lambda}\|b\|_{\mathrm{CBMO}^{\vec{q},\lambda}(\mathbb{R}^{n})}.
\end{align*}

If $k<j$, then
\begin{align*}
|b(y)-b_{B_{k}}|&\leq|b(y)-b_{B_{j}}|+|b_{B_{k}}-b_{B_{j}}|\\
&\leq|b(y)-b_{B_{j}}|+\sum^{j-1}_{i=k}|b_{B_{i}}-b_{B_{i+1}}|\\
&\leq|b(y)-b_{B_{j}}|+C2^{n\lambda}(1-2^{n\lambda})^{-1}(2^{kn\lambda}-2^{jn\lambda})\|b\|_{\mathrm{CBMO}^{\vec{q},\lambda}(\mathbb{R}^{n})}.
\end{align*}

If $j<k$, then
\begin{align*}
|b(y)-b_{B_{k}}|&\leq|b(y)-b_{B_{j}}|+\sum^{k-1}_{i=j}|b_{B_{i}}-b_{B_{i+1}}|\\
&\leq|b(y)-b_{B_{j}}|+C2^{n\lambda}(1-2^{n\lambda})^{-1}(2^{jn\lambda}-2^{kn\lambda})\|b\|_{\mathrm{CBMO}^{\vec{q},\lambda}(\mathbb{R}^{n})}.
\end{align*}
The proof is completed.
\end{proof}

Let us now give the proof of Theorem 3.1.

\begin{proof}[Proof]
Let $f$ be a function in $L^{\vec{p}}(\mathbb{R}^n)$, $b\in \mathrm{CBMO}^{\vec{q},\lambda}(\mathbb{R}^{n})\cap \mathrm{CBMO}^{\vec{p}',\lambda}(\mathbb{R}^{n})$. For simplicity, we write
$$\sum^{\infty}_{j=-\infty}f(x)\chi_{j}(x)=\sum^{\infty}_{j=-\infty}f_{j}(x),$$
and
$$\|\mathcal{H}_{b}f\|_{\vec{q}}=\sum^{\infty}_{k=-\infty}\|\chi_{B_{k}\backslash B_{k-1}}\mathcal{H}_{b}f\|_{\vec{q}}=\sum^{\infty}_{k=-\infty}\|\chi_{k}\mathcal{H}_{b}f\|_{\vec{q}}.$$
By the Minkowski inequality on mixed Lebesgue space, we have
\begin{align*}
\|\chi_{k}\mathcal{H}_{b}f\|_{\vec{q}}&=\Big\|\frac{\chi_{k}(\cdot)}{|\cdot|^{n}}\int_{B(0,|\cdot|)}\big(b(\cdot)-b(y)\big)f(y)dy\Big\|_{\vec{q}}\\
&\leq\Big\|\frac{\chi_{k}(\cdot)}{|\cdot|^{n}}\sum^{k}_{j=-\infty}\int_{C_{j}}\big|b(\cdot)-b(y)\big|\big|f(y)\big|dy\Big\|_{\vec{q}}\\
&\lesssim2^{-kn}\Big\|\chi_{k}(\cdot)\sum^{k}_{j=-\infty}\int_{C_{j}}\big|b(\cdot)-b(y)\big|\big|f(y)\big|dy\Big\|_{\vec{q}}\\
&\lesssim2^{-kn}\Big\|\chi_{k}(\cdot)\sum^{k}_{j=-\infty}\int_{C_{j}}\big|b(\cdot)-b_{B_{k}}\big|\big|f(y)\big|dy\Big\|_{\vec{q}}\\
&\quad+2^{-kn}\Big\|\chi_{k}(\cdot)\sum^{k}_{j=-\infty}\int_{C_{j}}\big|b(y)-b_{B_{k}}\big|\big|f(y)\big|dy\Big\|_{\vec{q}}\\
&:=I_{1}+I_{2}.
\end{align*}

We first estimate $I_{1}$. By the H\"{o}lder inequality on mixed Lebesgue space, we obtain
$$\int_{C_{j}}\big|b(x)-b_{B_{k}}\big|\big|f(y)\big|dy\leq\big|b(x)-b_{B_{k}}\big|\|\chi_{j}\|_{\vec{p}'}\|f_{j}\|_{\vec{p}},$$
then
\begin{align*}
I_{1}&\leq2^{-kn}\|(b-b_{B_{k}})\chi_{B_{k}}\|_{\vec{q}}\sum^{k}_{j=-\infty}\|\chi_{j}\|_{\vec{p}'}\|f_{j}\|_{\vec{p}}\\
&\leq2^{-kn(1-\lambda)}\|b\|_{\mathrm{CBMO}^{\vec{q},\lambda}(\mathbb{R}^{n})}\|\chi_{B_{k}}\|_{\vec{q}}\sum^{k}_{j=-\infty}\|\chi_{B_{j}}\|_{\vec{p}'}\|f_{j}\|_{\vec{p}}.
\end{align*}

Next we estimate $I_2$. It follows from Lemma 3.1 that
\begin{align*}
I_{2}&\lesssim2^{-kn}\Big\|\chi_{k}(\cdot)\sum^{k}_{j=-\infty}\int_{C_{j}}|b(y)-b_{B_{j}}||f(y)|dy\Big\|_{\vec{q}}\\
&\quad+2^{-kn}\|b\|_{\mathrm{CBMO}^{\vec{q},\lambda}(\mathbb{R}^{n})}\Big\|\chi_{k}(\cdot)\sum^{k}_{j=-\infty}
\int_{C_{j}}\frac{2^{n\lambda}}{1-2^{n\lambda}}(2^{jn\lambda}-2^{kn\lambda})|f(y)|dy\Big\|_{\vec{q}}\\
&:=I_{21}+I_{22}.
\end{align*}
The H\"{o}lder inequality on mixed Lebesgue space deduces that
\begin{align*}
I_{21}&\leq2^{-kn}\Big\|\chi_{B_{k}}(\cdot)\sum^{k}_{j=-\infty}\int_{C_{j}}|b(y)-b_{B_{j}}||f(y)|dy\Big\|_{\vec{q}}\\
&\leq2^{-kn}\|\chi_{B_{k}}\|_{\vec{q}}\sum^{k}_{j=-\infty}\|(b-b_{B_{j}})\chi_{B_{j}}\|_{\vec{p}'}\|f_{j}\|_{\vec{p}}\\
&\leq2^{-kn}\|b\|_{\mathrm{CBMO}^{\vec{p}',\lambda}(\mathbb{R}^{n})}\|\chi_{B_{k}}\|_{\vec{q}}\sum^{k}_{j=-\infty}2^{jn\lambda}\|\chi_{B_{j}}\|_{\vec{p}'}\|f_{j}\|_{\vec{p}},\\
\end{align*}
and
\begin{align*}
I_{22}&\leq2^{-kn}\|b\|_{\mathrm{CBMO}^{\vec{q},\lambda}(\mathbb{R}^{n})}\Big\|\chi_{B_{k}}(\cdot)
\sum^{k}_{j=-\infty}\frac{2^{n\lambda}}{1-2^{n\lambda}}(2^{jn\lambda}-2^{kn\lambda})\int_{C_{j}}|f(y)|dy\Big\|_{\vec{q}}\\
&\leq2^{-kn}\|b\|_{\mathrm{CBMO}^{\vec{q},\lambda}(\mathbb{R}^{n})}\Big\|\chi_{B_{k}}(\cdot)
\sum^{k}_{j=-\infty}\frac{2^{n\lambda}}{1-2^{n\lambda}}(2^{jn\lambda}-2^{kn\lambda})\|\chi_{j}\|_{\vec{p}'}\|f_{j}\|_{\vec{p}}\Big\|_{\vec{q}}\\
&\leq2^{-kn}\|b\|_{\mathrm{CBMO}^{\vec{q},\lambda}(\mathbb{R}^{n})}\|\chi_{B_{k}}\|_{\vec{q}}
\sum^{k}_{j=-\infty}\frac{2^{n\lambda}}{1-2^{n\lambda}}(2^{jn\lambda}-2^{kn\lambda})\|\chi_{B_{j}}\|_{\vec{p}'}\|f_{j}\|_{\vec{p}}.
\end{align*}
By the H\"{o}lder inequality, the estimates of $I_{1}$, $I_{21}$ and $I_{22}$ lead to
\begin{align*}
\|\mathcal{H}_{b}f\|_{\vec{q}}&=\sum^{\infty}_{k=-\infty}\|\chi_{k}\mathcal{H}_{b}f\|_{\vec{q}}\\
&\lesssim\|b\|_{\mathrm{CBMO}^{\vec{q},\lambda}(\mathbb{R}^{n})}\sum^{\infty}_{k=-\infty}2^{-kn(1-\lambda)}\|\chi_{B_{k}}\|_{\vec{q}}
\sum^{k}_{j=-\infty}\|\chi_{B_{j}}\|_{\vec{p}'}\|f_{j}\|_{\vec{p}}\\
&\quad+\|b\|_{\mathrm{CBMO}^{\vec{p}',\lambda}(\mathbb{R}^{n})}\sum^{\infty}_{k=-\infty}2^{-kn}\|\chi_{B_{k}}\|_{\vec{q}}
\sum^{k}_{j=-\infty}2^{jn\lambda}\|\chi_{B_{j}}\|_{\vec{p}'}\|f_{j}\|_{\vec{p}}\\
&\quad+\|b\|_{\mathrm{CBMO}^{\vec{q},\lambda}(\mathbb{R}^{n})}\sum^{\infty}_{k=-\infty}2^{-kn}\|\chi_{B_{k}}\|_{\vec{q}}
\sum^{k}_{j=-\infty}\frac{2^{n\lambda}(2^{jn\lambda}-2^{kn\lambda})}{1-2^{n\lambda}}\|\chi_{B_{j}}\|_{\vec{p}'}\|f_{j}\|_{\vec{p}}
\end{align*}
\begin{align*}
&\lesssim\big(\|b\|_{\mathrm{CBMO}^{\vec{p}',\lambda}(\mathbb{R}^{n})}+\|b\|_{\mathrm{CBMO}^{\vec{q},\lambda}(\mathbb{R}^{n})}\big)\sum^{\infty}_{k=-\infty}\|f_{k}\|_{\vec{p}}\\
&\quad+\|b\|_{\mathrm{CBMO}^{\vec{q},\lambda}(\mathbb{R}^{n})}\sum^{\infty}_{k=-\infty}2^{-kn(1-\frac{1}{n}\sum^{n}_{i=1}\frac{1}{q_{i}})}
\sum^{k}_{j=-\infty}\frac{2^{n\lambda}(2^{jn\lambda}-2^{kn\lambda})}{1-2^{n\lambda}}2^{jn\cdot\frac{1}{n}\sum^{n}_{i=1}\frac{1}{p_{i}'}}\|f_{j}\|_{\vec{p}}\\
&:=N_{1}+N_{2},
\end{align*}
where
$$\|\chi_{B_{k}}\|_{\vec{q}}\thickapprox2^{kn\cdot\frac{1}{n}\sum^{n}_{i=1}\frac{1}{q_{i}}},$$
and it follows from
$$\lambda+\frac{1}{n}\sum^{n}_{i=1}\frac{1}{p_{i}'}=1-\frac{1}{n}\sum^{n}_{i=1}\frac{1}{q_{i}}>0$$
that
\begin{align*}
\sum^{k}_{j=-\infty}2^{jn\lambda}\|\chi_{B_{j}}\|_{\vec{p}'}\|f_{j}\|_{\vec{p}}
&\leq[\sum^{k}_{j=-\infty}(2^{jn\lambda}\|\chi_{B_{j}}\|_{\vec{p}'})^{s}]^{\frac{1}{s}}[\sum^{k}_{j=-\infty}(\|f_{j}\|_{\vec{p}})^{s'}]^{\frac{1}{s'}}\\
&\leq(\sum^{k}_{j=-\infty}2^{jn\lambda}\|\chi_{B_{j}}\|_{\vec{p}'})(\sum^{k}_{j=-\infty}\|f_{j}\|_{\vec{p}})\\
&\leq\|f_{k}\|_{\vec{p}}\sum^{k}_{j=-\infty}2^{jn\lambda}\|\chi_{B_{j}}\|_{\vec{p}'}
\end{align*}
for any $1<s<\infty$ and $\frac{1}{s}+\frac{1}{s'}=1$.

For $N_{1}$, we have
\begin{align*}
N_{1}&=\big(\|b\|_{\mathrm{CBMO}^{\vec{p}',\lambda}(\mathbb{R}^{n})}+\|b\|_{\mathrm{CBMO}^{\vec{q},\lambda}(\mathbb{R}^{n})}\big)\sum^{\infty}_{k=-\infty}\|f\chi_{B_{k}\backslash B_{k-1}}\|_{\vec{p}}\\
&=\big(\|b\|_{\mathrm{CBMO}^{\vec{p}',\lambda}(\mathbb{R}^{n})}+\|b\|_{\mathrm{CBMO}^{\vec{q},\lambda}(\mathbb{R}^{n})}\big)\|f\|_{\vec{p}}.
\end{align*}
For $N_{2}$, using the H\"{o}lder inequality, we obtain
\begin{align*}
N_{2}&=\|b\|_{\mathrm{CBMO}^{\vec{q},\lambda}(\mathbb{R}^{n})}\sum^{\infty}_{k=-\infty}2^{-kn(1-\frac{1}{n}\sum^{n}_{i=1}\frac{1}{q_{i}})}\sum^{k}_{j=-\infty}
\frac{2^{n\lambda}(2^{jn\lambda}-2^{kn\lambda})}{1-2^{n\lambda}}\cdot2^{jn\cdot\frac{1}{n}\sum^{n}_{i=1}\frac{1}{p_{i}'}}\|f_{j}\|_{\vec{p}}\\
&\leq\|b\|_{\mathrm{CBMO}^{\vec{q},\lambda}(\mathbb{R}^{n})}\sum^{\infty}_{k=-\infty}\|f_{k}\|_{\vec{p}}\frac{2^{n\lambda}(2^{-kn(1-\frac{1}{n}\sum^{n}_{i=1}\frac{1}{q_{i}})})}{1-2^{n\lambda}}
\sum^{k}_{j=-\infty}(2^{jn\lambda}-2^{kn\lambda})2^{jn\cdot\frac{1}{n}\sum^{n}_{i=1}\frac{1}{p_{i}'}}\\
&\lesssim\|b\|_{\mathrm{CBMO}^{\vec{q},\lambda}(\mathbb{R}^{n})}\sum^{\infty}_{k=-\infty}\|f\chi_{B_{k}\backslash B_{k-1}}\|_{\vec{p}}\\
&\leq\big(\|b\|_{\mathrm{CBMO}^{\vec{p}',\lambda}(\mathbb{R}^{n})}+\|b\|_{\mathrm{CBMO}^{\vec{q},\lambda}(\mathbb{R}^{n})}\big)\|f\|_{\vec{p}},
\end{align*}
where we used the fact that
$$\lambda+\frac{1}{n}\sum^{n}_{i=1}\frac{1}{p_{i}'}=1-\frac{1}{n}\sum^{n}_{i=1}\frac{1}{q_{i}}>0.$$
Therefore, $\mathcal{H}_{b}$ is bounded from $L^{\vec{p}}(\mathbb{R}^{n})$ to $L^{\vec{q}}(\mathbb{R}^{n})$.

In a similar way, we can prove the boundedness of $\mathcal{H}^{*}_{b}$ from $L^{\vec{p}}(\mathbb{R}^{n})$ to $L^{\vec{q}}(\mathbb{R}^{n})$.

Conversely, suppose that $\mathcal{H}_{b}$ and $\mathcal{H}^{*}_{b}$ are bounded from $L^{\vec{p}}(\mathbb{R}^{n})$ to $L^{\vec{q}}(\mathbb{R}^{n})$, we will show that
$$b\in \mathrm{CBMO}^{\vec{p}',\lambda}(\mathbb{R}^{n})\cap \mathrm{CBMO}^{\vec{q},\lambda}(\mathbb{R}^{n}).$$

We consider only the situation when $b\in \mathrm{CBMO}^{\vec{q},\lambda}(\mathbb{R}^{n})$. Actually, we know that both $\mathcal{H}_{b}$ and $\mathcal{H}^{*}_{b}$ map $L^{\vec{q}'}(\mathbb{R}^{n})$ into $L^{\vec{p}'}(\mathbb{R}^{n})$ via Remark 3.1. Therefore, a similar procedure for all $b\in \mathrm{CBMO}^{\vec{p}',\lambda}(\mathbb{R}^{n})$.

For any fixed $r>0$, denote $B(0,r)$ by $B$. Since $\mathcal{H}_{b}$ and $\mathcal{H}^{*}_{b}$ are bounded from $L^{\vec{p}}(\mathbb{R}^{n})$ to $L^{\vec{q}}(\mathbb{R}^{n})$, we obtain
\begin{align*}
\frac{\|(b-b_{B(0,r)})\chi_{B(0,r)}\|_{\vec{q}}}{|B(0,r)|^{\lambda}\|\chi_{B(0,r)}\|_{\vec{q}}}&=
\frac{1}{|B|^{\lambda}\|\chi_{B}\|_{\vec{q}}}\Big\|\big(b(\cdot)-\frac{1}{|B|}\int_{B}b(z)dz\big)\chi_{B}(\cdot)\Big\|_{\vec{q}}\\
&=\frac{1}{|B|^{\lambda}\|\chi_{B}\|_{\vec{q}}}\Big\|\frac{\chi_{B}(\cdot)}{|B|}\int_{B}\big(b(\cdot)-b(z)\big)dz\Big\|_{\vec{q}}\\
&\leq\frac{1}{|B|^{\lambda}\|\chi_{B}\|_{\vec{q}}}\Big\|\frac{\chi_{B}(\cdot)}{|B|}\int_{|z|<|\cdot|}\big(b(\cdot)-b(z)\big)\chi_{B}(z)dz\Big\|_{\vec{q}}\\
&\quad+\frac{1}{|B|^{\lambda}\|\chi_{B}\|_{\vec{q}}}\Big\|\frac{\chi_{B}(\cdot)}{|B|}\int_{|z|\geq|\cdot|}\big(b(\cdot)-b(z)\big)\chi_{B}(z)dz\Big\|_{\vec{q}}\\
&\leq\frac{1}{|B|^{\lambda}\|\chi_{B}\|_{\vec{q}}}\Big\|\frac{\chi_{B}(\cdot)}{|B|}|\cdot|^{n}\frac{1}{|\cdot|^{n}}\int_{|z|<|\cdot|}\big(b(\cdot)-b(z)\big)\chi_{B}(z)dz\Big\|_{\vec{q}}\\
&\quad+\frac{1}{|B|^{\lambda}\|\chi_{B}\|_{\vec{q}}}\Big\|\frac{\chi_{B}(\cdot)}{|B|}\int_{|z|\geq|\cdot|}\frac{\big(b(\cdot)-b(z)\big)\chi_{B}(z)}{|z|^{n}}|z|^{n}dz\Big\|_{\vec{q}}\\
&\leq\frac{C}{|B|^{\lambda}\|\chi_{B}\|_{\vec{q}}}\Big\|\chi_{B}(\cdot)\mathcal{H}_{b}(\chi_{B})(\cdot)\Big\|_{\vec{q}}
+\frac{C}{|B|^{\lambda}\|\chi_{B}\|_{\vec{q}}}\Big\|\chi_{B}(\cdot)\mathcal{H}^{*}_{b}(\chi_{B})(\cdot)\Big\|_{\vec{q}}\\
&\leq\frac{C}{|B|^{\lambda}\|\chi_{B}\|_{\vec{q}}}\Big\|\mathcal{H}_{b}(\chi_{B})\Big\|_{\vec{q}}
+\frac{C}{|B|^{\lambda}\|\chi_{B}\|_{\vec{q}}}\Big\|\mathcal{H}^{*}_{b}(\chi_{B})\Big\|_{\vec{q}}\\
&\leq\frac{C}{|B|^{\lambda}\|\chi_{B}\|_{\vec{q}}}\|\chi_{B}\|_{\vec{p}}\\
&\leq C,
\end{align*}
where we used the condition
$$\lambda=\frac{1}{n}\sum^{n}_{i=1}\frac{1}{p_{i}}-\frac{1}{n}\sum^{n}_{i=1}\frac{1}{q_{i}}.$$
The proof of Theorem 3.1 is finished.
\end{proof}

\section{The $\mathrm{CBMO}^{\vec{q},\lambda}(\mathbb{R}^{n})$ estimates of the Hardy operators $\mathcal{H}$ and $\mathcal{H}^{*}$ on $\mathcal{B}^{\vec{q},\lambda}(\mathbb{R}^{n})$}
In this section, we will establish the boundedness of commutators $\mathcal{H}_{b}f$ and $\mathcal{H}^{*}_{b}f$ of the Hardy operator and its dual operator generated with mixed $\lambda$-central $\mathrm{BMO}$ function $b$ on mixed $\lambda$-central Morrey space $\mathcal{B}^{\vec{q},\lambda}(\mathbb{R}^{n})$, respectively. Our main results can be stated as follows.

\begin{theorem}
Let $1<\vec{p_{1}},\vec{q}<\infty$, $\vec{p_{1}}'<\vec{p_{2}}<\infty$, $\frac{1}{\vec{q}}=\frac{1}{\vec{p_{1}}}+\frac{1}{\vec{p_{2}}}$, $-\frac{1}{n}\sum^{n}_{i=1}\frac{1}{q_{i}}<\lambda<0$, $0\leq\lambda_{2}<\frac{1}{n}$ and $\lambda=\lambda_{1}+\lambda_{2}$. If $b\in \mathrm{CBMO}^{\vec{p_{2}},\lambda_{2}}(\mathbb{R}^{n})$, then $\mathcal{H}_{b}$ is bounded from $\mathcal{B}^{\vec{p_{1}},\lambda_{1}}(\mathbb{R}^{n})$ to $\mathcal{B}^{\vec{q},\lambda}(\mathbb{R}^{n})$ satisfies the following inequality:
$$\|\mathcal{H}_{b}f\|_{\mathcal{B}^{\vec{q},\lambda}(\mathbb{R}^{n})}\lesssim\|b\|_{\mathrm{CBMO}^{\vec{p_{2}},\lambda_{2}}(\mathbb{R}^{n})}
\|f\|_{\mathcal{B}^{\vec{p_{1}},\lambda_{1}}(\mathbb{R}^{n})}.$$
\end{theorem}

\begin{theorem}
Let $1<\vec{p_{1}},\vec{q}<\infty$, $\vec{p_{1}}'<\vec{p_{2}}<\infty$, $\frac{1}{\vec{q}}=\frac{1}{\vec{p_{1}}}+\frac{1}{\vec{p_{2}}}$, $-\frac{1}{n}\sum^{n}_{i=1}\frac{1}{q_{i}}<\lambda<0$, $0\leq\lambda_{2}<\frac{1}{n}$ and $\lambda=\lambda_{1}+\lambda_{2}$. If $b\in \mathrm{CBMO}^{\vec{p_{2}},\lambda_{2}}(\mathbb{R}^{n})$, then $\mathcal{H}^{*}_{b}$ is bounded from $\mathcal{B}^{\vec{p_{1}},\lambda_{1}}(\mathbb{R}^{n})$ to $\mathcal{B}^{\vec{q},\lambda}(\mathbb{R}^{n})$ satisfies the following inequality:
$$\|\mathcal{H}^{*}_{b}f\|_{\mathcal{B}^{\vec{q},\lambda}(\mathbb{R}^{n})}\lesssim\|b\|_{\mathrm{CBMO}^{\vec{p_{2}},\lambda_{2}}(\mathbb{R}^{n})}
\|f\|_{\mathcal{B}^{\vec{p_{1}},\lambda_{1}}(\mathbb{R}^{n})}.$$
\end{theorem}

Now we give the proofs of Theorems 4.1 and 4.2 in this position. In fact, since $\mathcal{H}$ and $\mathcal{H}^{*}$ are dual operators to each other, the methods of the proofs are standard, so we only provide proof of Theorem 4.1 with the following.

\begin{proof}
Let $f$ be a function in $\mathcal{B}^{\vec{p_{1}},\lambda_{1}}(\mathbb{R}^{n})$. For fixed $R>0$, denote $B(0,R)$ by $B$. We need to check that the fact
$$\|\mathcal{H}_{b}f\chi_{B}\|_{\vec{q}}\lesssim|B|^{\lambda}\|\chi_{B}\|_{\vec{q}}\|b\|_{\mathrm{CBMO}^{\vec{p_{2}},\lambda_{2}}(\mathbb{R}^{n})}
\|f\|_{\mathcal{B}^{\vec{p_{1}},\lambda_{1}}(\mathbb{R}^{n})}$$
holds. The Minkowski inequality on mixed Lebesgue space gives that
\begin{align*}
\|\mathcal{H}_{b}f(\cdot)\chi_{B}(\cdot)\|_{\vec{q}}&=\Big\|\chi_{B}(\cdot)\frac{1}{|\cdot|^{n}}\int_{|y|<|\cdot|}\big(b(\cdot)-b(y)\big)f(y)dy\Big\|_{\vec{q}}\\
&\leq\Big\|\chi_{B}(\cdot)\frac{1}{|\cdot|^{n}}\int_{|y|<|\cdot|}\big(b(\cdot)-b_{B}\big)f(y)dy\Big\|_{\vec{q}}\\
&\quad+\Big\|\chi_{B}(\cdot)\frac{1}{|\cdot|^{n}}\int_{|y|<|\cdot|}\big(b(y)-b_{B}\big)f(y)dy\Big\|_{\vec{q}}\\
&=\Big\|\chi_{B}\big(b-b_{B}\big)\mathcal{H}(f\chi_{B})\Big\|_{\vec{q}}\\
&\quad+\Big\|\chi_{B}(\cdot)\frac{1}{|\cdot|^{n}}\int_{|y|<|\cdot|}\big(b(y)-b_{B}\big)f(y)dy\Big\|_{\vec{q}}\\
&:=M_{1}+M_{2}.
\end{align*}

First, we estimate $M_{1}$. For $\frac{1}{\vec{q}}=\frac{1}{\vec{p_{1}}}+\frac{1}{\vec{p_{2}}}$, by the H\"{o}lder inequality on mixed Lebesgue space and the boundedness of $\mathcal{H}$ from $L^{\vec{p_{1}}}(\mathbb{R}^{n})$ to $L^{\vec{p_{1}}}(\mathbb{R}^{n})$ in Lemma 2.1, we get
\begin{align*}
M_{1}&\leq\|\chi_{B}\big(b-b_{B}\big)\|_{\vec{p_{2}}}\|\mathcal{H}(f\chi_{B})\|_{\vec{p_{1}}}\\
&\leq|B|^{\lambda_{2}}\|\chi_{B}\|_{\vec{p_{2}}}\|b\|_{\mathrm{CBMO}^{\vec{p_{2}},\lambda_{2}}(\mathbb{R}^{n})}\|f\chi_{B}\|_{\vec{p_{1}}}\\
&\leq|B|^{\lambda}\|\chi_{B}\|_{\vec{q}}\|b\|_{\mathrm{CBMO}^{\vec{p_{2}},\lambda_{2}}(\mathbb{R}^{n})}\|f\|_{\mathcal{B}^{\vec{p_{1}},\lambda_{1}}(\mathbb{R}^{n})},
\end{align*}
where we used the condition $\lambda=\lambda_{1}+\lambda_{2}$.

Next, we estimate $M_{2}$. By the Minkowski inequality on mixed Lebesgue space, we have
\begin{align*}
M_{2}&=\Big\|\chi_{B}(\cdot)\frac{1}{|\cdot|^{n}}\int_{B(0,|\cdot|)}\big(b(y)-b_{B}\big)f(y)dy\Big\|_{\vec{q}}\\
&=\sum^{0}_{k=-\infty}\Big\|\chi_{2^{k}B\backslash2^{k-1}B}(\cdot)\frac{1}{|\cdot|^{n}}\int_{B(0,|\cdot|)}\big(b(y)-b_{B}\big)f(y)dy\Big\|_{\vec{q}}
\end{align*}
\begin{align*}
&\lesssim\sum^{0}_{k=-\infty}\frac{1}{|2^{k}B|}\Big\|\chi_{2^{k}B\backslash2^{k-1}B}(\cdot)\sum^{k}_{j=-\infty}\int_{2^{j}B\backslash2^{j-1}B}
|b(y)-b_{B}||f(y)|dy\Big\|_{\vec{q}}\\
&\leq\sum^{0}_{k=-\infty}\frac{1}{|2^{k}B|}\Big\|\chi_{2^{k}B\backslash2^{k-1}B}(\cdot)\sum^{k}_{j=-\infty}\int_{2^{j}B\backslash2^{j-1}B}
|b(y)-b_{2^{j}B}||f(y)|dy\Big\|_{\vec{q}}\\
&\quad+\sum^{0}_{k=-\infty}\frac{1}{|2^{k}B|}\Big\|\chi_{2^{k}B\backslash2^{k-1}B}(\cdot)\sum^{k}_{j=-\infty}\int_{2^{j}B\backslash2^{j-1}B}
|b_{2^{j}B}-b_{B}||f(y)|dy\Big\|_{\vec{q}}\\
&:=M_{21}+M_{22}.
\end{align*}

For $M_{21}$, it follows from $\frac{1}{\vec{q}}=\frac{1}{\vec{p_{1}}}+\frac{1}{\vec{p_{2}}}$, $1=\frac{1}{\vec{p_{1}}}+\frac{1}{\vec{p_{1}}'}$ and the H\"{o}lder inequality on mixed Lebesgue space, we obtain
\begin{align*}
M_{21}&\leq\sum^{0}_{k=-\infty}\frac{1}{|2^{k}B|}\Big\|\chi_{2^{k}B\backslash2^{k-1}B}(\cdot)\sum^{k}_{j=-\infty}\int_{2^{j}B}
|b(y)-b_{2^{j}B}||f(y)|dy\Big\|_{\vec{q}}\\
&\leq\sum^{0}_{k=-\infty}\frac{1}{|2^{k}B|}\|\chi_{2^{k}B}\|_{\vec{q}}\sum^{k}_{j=-\infty}\|(b-b_{2^{j}B})\chi_{2^{j}B}\|_{\vec{p_{1}}'}\|f\chi_{2^{j}B}\|_{\vec{p_{1}}}\\
&\leq\sum^{0}_{k=-\infty}\frac{1}{|2^{k}B|}\|\chi_{2^{k}B}\|_{\vec{q}}\sum^{k}_{j=-\infty}\|(b-b_{2^{j}B})\chi_{2^{j}B}\|_{\vec{p_{2}}}\|\chi_{2^{j}B}\|_{\vec{q}'}
|2^{j}B|^{\lambda_{1}}\|\chi_{2^{j}B}\|_{\vec{p_{1}}}\|f\|_{\mathcal{B}^{\vec{p_{1}},\lambda_{1}}(\mathbb{R}^{n})}\\
&\leq\sum^{0}_{k=-\infty}\frac{1}{|2^{k}B|}\|\chi_{2^{k}B}\|_{\vec{q}}\sum^{k}_{j=-\infty}|2^{j}B|^{\lambda+1}
\|b\|_{\mathrm{CBMO}^{\vec{p_{2}},\lambda_{2}}(\mathbb{R}^{n})}\|f\|_{\mathcal{B}^{\vec{p_{1}},\lambda_{1}}(\mathbb{R}^{n})}\\
&\leq|B|^{\lambda}\|\chi_{B}\|_{\vec{q}}\|b\|_{\mathrm{CBMO}^{\vec{p_{2}},\lambda_{2}}(\mathbb{R}^{n})}
\|f\|_{\mathcal{B}^{\vec{p_{1}},\lambda_{1}}(\mathbb{R}^{n})}\sum^{0}_{k=-\infty}2^{kn(\lambda+\frac{1}{n}\sum^{n}_{i=1}\frac{1}{q_{i}})}\\
&\lesssim|B|^{\lambda}\|\chi_{B}\|_{\vec{q}}\|b\|_{\mathrm{CBMO}^{\vec{p_{2}},\lambda_{2}}(\mathbb{R}^{n})}
\|f\|_{\mathcal{B}^{\vec{p_{1}},\lambda_{1}}(\mathbb{R}^{n})},
\end{align*}
where
$$\|\chi_{2^{j}B}\|_{\vec{p_{1}}}\|\chi_{2^{j}B}\|_{\vec{p_{2}}}\|\chi_{2^{j}B}\|_{\vec{q}'}\thickapprox
|2^{j}B|^{\frac{1}{n}\sum^{n}_{i=1}\frac{1}{p_{1i}}+\frac{1}{n}\sum^{n}_{i=1}\frac{1}{p_{2i}}+\frac{1}{n}\sum^{n}_{i=1}\frac{1}{q_{i}'}}=|2^{j}B|,$$
and we used the fact that $-\frac{1}{n}\sum^{n}_{i=1}\frac{1}{q_{i}}<\lambda<0$.

To estimate $M_{22}$, the following fact is applied. For $\lambda_{2}\geq0$, by the H\"{o}lder inequality on mixed Lebesgue space $(1=\frac{1}{\vec{p_{2}}}+\frac{1}{\vec{p_{2}}'})$, we have
\begin{align*}
|b_{2^{j}B}-b_{B}|&\leq\sum^{-1}_{i=j}|b_{2^{i+1}B}-b_{2^{i}B}|\\
&\leq\sum^{-1}_{i=j}\frac{1}{|2^{i}B|}\int_{2^{i}B}|b(y)-b_{2^{i+1}B}|dy\\
&\leq\sum^{-1}_{i=j}\frac{1}{|2^{i}B|}\|(b-b_{2^{i+1}B})\chi_{2^{i+1}B}\|_{\vec{p_{2}}}\|\chi_{2^{i+1}B}\|_{\vec{p_{2}}'}\\
&\leq\|b\|_{\mathrm{CBMO}^{\vec{p_{2}},\lambda_{2}}(\mathbb{R}^{n})}\sum^{-1}_{i=j}\frac{1}{|2^{i}B|}|2^{i+1}B|^{\lambda_{2}}\|\chi_{2^{i+1}B}\|_{\vec{p_{2}}}
\|\chi_{2^{i+1}B}\|_{\vec{p_{2}}'}
\end{align*}
\begin{align*}
&\lesssim|B|^{\lambda_{2}}\|b\|_{\mathrm{CBMO}^{\vec{p_{2}},\lambda_{2}}(\mathbb{R}^{n})}\sum^{-1}_{i=j}2^{(i+1)n\lambda_{2}}\\
&\lesssim|j||B|^{\lambda_{2}}\|b\|_{\mathrm{CBMO}^{\vec{p_{2}},\lambda_{2}}(\mathbb{R}^{n})}.
\end{align*}

Therefore, for $-\frac{1}{n}\sum^{n}_{i=1}\frac{1}{q_{i}}<\lambda<0$, $0\leq\lambda_{2}<\frac{1}{n}$ and $\lambda=\lambda_{1}+\lambda_{2}$, by the H\"{o}lder inequality $(1=\frac{1}{\vec{p_{1}}}+\frac{1}{\vec{p_{1}}'})$ on mixed Lebesgue space, we obtain
\begin{align*}
M_{22}&\lesssim|B|^{\lambda_{2}}\|b\|_{\mathrm{CBMO}^{\vec{p_{2}},\lambda_{2}}(\mathbb{R}^{n})}\sum^{0}_{k=-\infty}\frac{1}{|2^{k}B|}
\Big\|\chi_{2^{k}B\backslash2^{k-1}B}(\cdot)\sum^{k}_{j=-\infty}|j|\int_{2^{j}B}|f(y)|dy\Big\|_{\vec{q}}\\
&\leq|B|^{\lambda_{2}}\|b\|_{\mathrm{CBMO}^{\vec{p_{2}},\lambda_{2}}(\mathbb{R}^{n})}\sum^{0}_{k=-\infty}\frac{1}{|2^{k}B|}
\|\chi_{2^{k}B}\|_{\vec{q}}\sum^{k}_{j=-\infty}|j|\cdot\|f\chi_{2^{j}B}\|_{\vec{p_{1}}}\|\chi_{2^{j}B}\|_{\vec{p_{1}}'}\\
&\leq|B|^{\lambda_{2}}\|b\|_{\mathrm{CBMO}^{\vec{p_{2}},\lambda_{2}}(\mathbb{R}^{n})}\|f\|_{\mathcal{B}^{\vec{p_{1}},\lambda_{1}}(\mathbb{R}^{n})}
\sum^{0}_{k=-\infty}\frac{1}{|2^{k}B|}\|\chi_{2^{k}B}\|_{\vec{q}}\sum^{k}_{j=-\infty}|j||2^{j}B|^{\lambda_{1}+1}\\
&\lesssim|B|^{\lambda}\|\chi_{B}\|_{\vec{q}}\|b\|_{\mathrm{CBMO}^{\vec{p_{2}},\lambda_{2}}(\mathbb{R}^{n})}
\|f\|_{\mathcal{B}^{\vec{p_{1}},\lambda_{1}}(\mathbb{R}^{n})}\sum^{0}_{k=-\infty}|k|2^{kn(\lambda_{1}+\frac{1}{n}\sum^{n}_{i=1}\frac{1}{q_{i}})}\\
&\lesssim|B|^{\lambda}\|\chi_{B}\|_{\vec{q}}\|b\|_{\mathrm{CBMO}^{\vec{p_{2}},\lambda_{2}}(\mathbb{R}^{n})}
\|f\|_{\mathcal{B}^{\vec{p_{1}},\lambda_{1}}(\mathbb{R}^{n})}.
\end{align*}
Combining the estimates of $M_{21}$ and $M_{22}$, we deduce that
$$M_{2}\lesssim|B|^{\lambda}\|\chi_{B}\|_{\vec{q}}\|b\|_{\mathrm{CBMO}^{\vec{p_{2}},\lambda_{2}}(\mathbb{R}^{n})}
\|f\|_{\mathcal{B}^{\vec{p_{1}},\lambda_{1}}(\mathbb{R}^{n})}.$$
Thus, $M_{1}$ together with $M_{2}$ implies that
$$\|\mathcal{H}_{b}f\chi_{B}\|_{\vec{q}}\lesssim|B|^{\lambda}\|\chi_{B}\|_{\vec{q}}\|b\|_{\mathrm{CBMO}^{\vec{p_{2}},\lambda_{2}}(\mathbb{R}^{n})}
\|f\|_{\mathcal{B}^{\vec{p_{1}},\lambda_{1}}(\mathbb{R}^{n})}.$$
This completes the proof of Theorem 4.1.
\end{proof}

\noindent$\textbf{Declarations.}$\\
The authors declare that there is no conflict of interests regarding the publication of this paper.
\noindent$\textbf{Acknowledgement.}$ \\
The authors would like to express their gratitude to the referee for his/her very valuable comments.

\bigskip

\noindent Wenna Lu

\medskip

\noindent College of Mathematics and System Sciences,\ Xinjiang University\\
\"{U}r\"{u}mqi  830046,\ People's Republic of China

\smallskip

\noindent{\it E-mail address}: \texttt{luwnmath@126.com}

\bigskip

\noindent Jiang Zhou$^\ast$ (Corresponding author)

\medskip

\noindent College of Mathematics and System Sciences,\ Xinjiang University\\
\"{U}r\"{u}mqi  830046,\ People's Republic of China

\smallskip
\noindent{\it E-mail address}: \texttt{zhoujiang@xju.edu.cn}


\begin{thebibliography}{99}
\bibitem{AI}{N. Antonic and I. Ivec},
\textit{On the H\"{o}rmander-Mihlin theorem for mixed-norm Lebesgue spaces.}
J. Math. Anal. Appl., {\bf 433}(1) (2016), 176-199.

\bibitem{AL}{J. Alvarez, J. Lakey and M. Guzm\'{a}n-Partida},
\textit{Spaces of bounded $\lambda$-central mean oscillation, Morrey spaces, and $\lambda$-central Carleson measures.}
Collect. Math., {\bf 51}(1) (2000), 1-47.

\bibitem{BP}{A. Benedek and R. Panzone},
\textit{The space $L^{p}$, with mixed norm.}
Duke Math. J., {\bf 28} (1961), 301-324.

\bibitem{CG}{M. Christ and L. Grafakos},
\textit{Best constants for two nonconvolution inequalities.}
Proc. Amer. Math. Soc., {\bf 123}(6) (1995), 1687-1693.

\bibitem{CRW}{R. R. Coifman, R. Rochberg and G. Weiss},
\textit{Factorization theorems for Hardy spaces in several variables.}
Ann. of Math., {\bf 103}(3) (1976), 611-635.

\bibitem{DKP}{H. J. Dong, D. Kim and T. Phan},
\textit{Boundary Lebesgue mixed-norm estimates for non-stationary Stokes systems with $VMO$ coefficients.}
Comm. PDE., {\bf 47}(8) (2022), 1700-1731.

\bibitem{F1}{W. G. Faris},
\textit{Weak Lebesgue spaces and quantum mechanical binding.}
Duke Math. J., {\bf 43}(2) (1976), 365-373.

\bibitem{FLW}{Z. W. Fu, Z. G. Liu, S. Z. Lu and H. B. Wang},
\textit{Characterization for commutators of n-dimensional fractional Hardy operators.}
Sci. China Ser., {\bf A50}(10) (2007), 1418-1426.

\bibitem{FLL}{Z. W. Fu, Y. Lin and S. Z. Lu},
\textit{$\lambda$-central $\mathrm{BMO}$ estimates for commutators of singular integral operators with rough kernels.}
Acta Math. Sin., {\bf 24}(3) (2008), 373-386.

\bibitem{FL}{Z. W. Fu and S. Z. Lu},
\textit{Commutators of generalized Hardy operators.}
Math. Nachr., {\bf 282}(6) (2009), 832-845.

\bibitem{H1}{G. H. Hardy},
\textit{Note on a theorem of Hilbert.}
Math. Z., {\bf 6}(3-4) (1920), 314-317.

\bibitem{HLP}{G. H. Hardy, J. E. Littlewood and G. P\'{o}lya},
\textit{Inequalities.}
Reprint of the 1952 edition. Cambridge Mathematical Library. Cambridge University Press, Cambridge, 1988.

\bibitem{KC}{C. E. Kenig},
\textit{On the local and global well-posedness theory for the KP-I equation.}
Ann. Inst. H. Poincar¨¦ C Anal. Non Linaire, {\bf 21}(6) (2004), 827-838.

\bibitem{KD}{D. Kim},
\textit{Elliptic and parabolic equations with measurable coefficients in $L^{p}$-spaces with mixed norms.}
Methods Appl. Anal., {\bf 15}(4) (2008), 437-467.

\bibitem{KN}{N. V. Krylov},
\textit{Parabolic equations with $VMO$ coefficients in Sobolev spaces with mixed norms.}
J. Funct. Anal., {\bf 250}(2) (2007), 521-558.

\bibitem{KY}{Y. Komori},
\textit{Notes on commutators of Hardy operators.}
Int. J. Pure Appl. Math., {\bf 7}(3) (2003), 329-334.

\bibitem{LW}{S. C. Long and J. Wang},
\textit{Commutators of Hardy operators.}
J. Math. Anal. Appl., {\bf 274}(2) (2002), 626-644.

\bibitem{LZ}{W. N. Lu and J. Zhou},
\textit{Fractional integral operators on the mixed $\lambda$-central Morrey spaces.}
Southeast Asian Bull. Math.,(2023), (in press).

\bibitem{NT1}{T. Nogayama},
\textit{Mixed Morrey spaces.}
Positivity, {\bf 23}(4) (2019), 961-1000.

\bibitem{NT2}{T. Nogayama},
\textit{Boundedness of commutators of fractional integral operators on mixed Morrey spaces.}
Integral Transform Spec. Funct., {\bf 30}(10) (2019), 790-816.

\bibitem{SL}{S. G. Shi and S. Z. Lu},
\textit{Characterization of the central Campanato space via the commutator operator of Hardy type.}
J. Math. Anal. Appl., {\bf 429}(2) (2015), 713-732.

\bibitem{W1}{M. Q. Wei},
\textit{A characterization of $C\dot{M}O^{\vec{q}}$ via the commutator of Hardy-type operators on mixed Herz spaces.}
Appl. Anal., {\bf 101}(16) (2022), 5727-5742.

\bibitem{YT}{X. Yu and X. Tao},
\textit{Boundedness for a class of generalized commutators on $\lambda$-central Morrey spaces.}
Acta Math. Sin., {\bf 29}(10) (2013), 1917-1926.

\bibitem{ZFL}{F. Y. Zhao, Z. W. Fu and S. Z. Lu},
\textit{Endpoint estimates for n-dimensional Hardy operators and their commutators.}
Sci. China Math., {\bf 55}(10) (2012), 1977-1990.

\bibitem{ZL}{F. Y. Zhao and S. Z. Lu},
\textit{A characterization of $\lambda$-central $BMO$ space.}
Front. Math. China, {\bf 8}(1) (2013), 229-238.

\end{thebibliography}
\end{document}